\documentclass{tran-l}

\newtheorem{theorem}{Theorem}[section]
\newtheorem{lemma}[theorem]{Lemma}

\theoremstyle{definition}
\newtheorem{definition}[theorem]{Definition}
\newtheorem{example}[theorem]{Example}

\newtheorem{corollary}{Corollary}
\theoremstyle{remark}

\numberwithin{equation}{section}

\begin{document}

\title[Hindman's theorem]{A recursive coloring function without $\Pi_3^0$ solution for Hindman's theorem\footnote{This research was partially supported by MOE2019-T2-2-121, R146-000-337-114 / A-0008494-00-00 and R252-000-C17-114 / A-0008454-00-00.}}


\author{Yuke Liao}
\address{}
\curraddr{}
\email{liao\_yuke@u.nus.edu}
\thanks{}


\subjclass[2020]{Primary }

\date{}

\dedicatory{}

\begin{abstract}
	We show that there exists a recursive coloring function $c$ such that any $\Pi^0_3$ set is not a solution to $c$ for Hindman's theorem. We also show that there exists a recursive coloring function $c$ such that any $\Delta^0_3$ set is not a solution to $c$ for Hindman's theorem restricted on sum of at most three numbers. 
\end{abstract}

\maketitle

\section{Introduction}

	Reverse mathematics is an area of mathematical logic trying to find the strength of theorems: which axioms are needed to prove these theorems. Among all axiom systems, reverse mathematics mainly focuses on five subsystems of second-order arithmetic, known as the Big Five: $$\text{\rm{RCA}}_0,\text{\rm{WKL}}_0,\text{\rm{ACA}}_0,\text{\rm{ATR}}_0,\Pi^1_1\text{\rm{-CA}}_0.$$ This paper will mention $\text{\rm{ACA}}_0$, which is equivalent to say that the Turing jump exists for every set of natural numbers.

One of the most important topics in reverse mathematics is the Ramsey's theorem and its variants. The reverse mathematics version of Ramsey's theorem can be seen as an infinite version of Ramsey's theorem in combinatorics. This theorem states that, for any coloring (of finite many colors) function $c$ on $[\mathbb{N}]^n$, there exists an infinite set that is homogeneous. This infinite set is called a solution to $c$.

The strategy of analyzing these theorems is to study the computational complexity of their solutions. Jockusch \cite{Jockusch1972} proved that the strength of Ramsey's theorem (for $n\geq 3$) is exactly $\text{\rm{ACA}}_0$ by examining the complexity of the solutions: every recursive coloring has a $\Pi_2^0$ solution and therefore $\text{\rm{ACA}}_0$ is sufficient to prove Ramsey's theorem; and there exists a coloring such that any solution to it can compute its Turing jump, which shows $\text{\rm{ACA}}_0$ is also necessary.

In this paper, we discuss about the reverse mathematics about Hindman's theorem. Hindman's theorem is a Ramsey-type combinatorics theorem about finite sums:

\begin{theorem}[Hindman, 1974\cite{HINDMAN19741}]
		For any coloring function $c:\mathbb{N}\rightarrow r$, where $r\geq 2$ is a natural number, there exists an infinite set $H\subset\mathbb{N}$ such that $FS(H)$ is monochromatic under $c$ where $FS(H)$ is the set consisting of all finite sums of different numbers from $H$:
\begin{equation*}
	FS(H)=\left\{\sum_{x\in F}x:F\subset H, F\text{ is finite and nonempty}\right\}.
\end{equation*}

\end{theorem}

For convenience, we call such a set $H$ in the above theorem a solution to $c$ for Hindman's theorem. 	In this paper, Hindman's theorem is abbreviated as $\text{\rm{HT}}$.

The study of reverse mathematics of Hindman's theorem began with the paper by Blass, Hirst, and Simpson in 1987 \cite{bhs}. They established some basic results about complexity of solutions of Hindman's theorem:

\begin{theorem}[Blass, Hirst, and Simpson 1987\cite{bhs}]
		1.  For any coloring function $c:\mathbb{N}\rightarrow r$ , there exists a solution to $c$ for Hindman's theorem that is recursive in $c^{(\omega+1)}$, the $(\omega+1)$-th Turing jump of $c$. 
	
	2. There exists a recursive coloring function $c$ such that any solution to $c$ can compute $c'$; this result can be relativized to any degree. 
\end{theorem}

 	Based on these results about the complexity of solutions, they showed that the strength of Hindman's theorem is between $\rm{ACA_0}$ and $\rm{ACA_0^+}$: $\rm{ACA_0}$ is necessary and $\rm{ACA_0^+}$ is sufficient. Simply speaking,  $\rm{ACA_0^+}$ is to say the $\omega$-th Turing jump of every set exists. However, this result has not been improved since 1987.

 In addition, before presenting their main results, Blass, Hirst, and Simpson provided another theorem as an example, which gives a lower bound on the complexity of a solution for Hindman's theorem. That is:
 
 \begin{theorem}[Blass, Hirst, and Simpson 1987\cite{bhs}]\label{lem:ch1s1:improvement}
 	There exists a recursive coloring function $c$ in two colors without any $\Delta^0_2$ solutions for Hindman's theorem. 
 \end{theorem}

We may call Theorem \ref{lem:ch1s1:improvement} the $\Delta_2^0$-unsolvability of Hindman's theorem. In the paper \cite{towsner}, Towsner gave a simplified proof of Hindman theorem using the techniques of half-match and full-match  of a coloring function. Towsner showed that every recursive coloring function has a full-match which is an arithmetic set and hence a solution of $\emptyset^{(\omega+1)}$ degree. Moreover, if every recursive coloring has a full-match which is a recursive set, then every recursive coloring will have an arithmetic solution for Hindman's theorem and $\text{ACA}_0$ will be the exact strength of Hindman's theorem. In the same paper, Towsner also discussed about the improvement of Theorem \ref{lem:ch1s1:improvement}.

In this paper, we will prove the $\Pi_3^0$-unsolvability of Hindman's theorem:

\begin{theorem}
	There exists a recursive coloring function $c$ in finitely many colors without any $\Pi_3^0$ solutions for Hindman's theorem. 
\end{theorem}

Another topic of Hindman's theorem is the restricted version of Hindman's theorem. The most common restricted Hindman theorem's is the following theorem $\text{\rm{HT}}^{\leq k}_r$:

\begin{theorem}[restricted Hindman's theorem $\text{\rm{HT}}^{\leq k}_r$ for fixed numbers $r,k$]
	For any coloring function $c:\mathbb{N}\rightarrow r$ and any natural number $k$, there exists an infinite set $H\subset \mathbb{N}$ such that $FS^{\leq k}(H)$ is monochromatic under $c$. Here, $FS^{\leq k}(H)$ is defined as following: 	
	\begin{equation*}
		FS^{\leq k}(H)=\left\{\sum_{x\in F}x:F\subset H, |F|\leq k\text{ is and nonempty}\right\}.
	\end{equation*}
\end{theorem}

There are plenty of results about the restricted Hindman theorems  $\text{\rm{HT}}^{\leq k}_r$, see \cite{Dzhafarov20172} and \cite{NewBoundsofHindman}. However, even for such weak theorems, we do not know much about their strength. For example, we do not know if $\rm{ACA_0}$ can prove any such restricted Hindman's theorem, even for $\text{\rm{HT}}^{\leq 2}_2$.

In this paper, we will not focus on the numbers of colors $r$ and will consider the following restricted version written as $\text{\rm{HT}}^{\leq k}$:

\begin{theorem}[restricted Hindman's theorem $\text{\rm{HT}}^{\leq k}$ for a fixed number $k$]
	For any coloring function $c:\mathbb{N}\rightarrow r$, where $r\geq 2$ is a natural number , there exists an infinite set $H\subset \mathbb{N}$ such that $FS^{\leq k}(H)$ is monochromatic. 
\end{theorem}

In Section 2, we will prove the $\Delta_3^0$-unsolvability for $\text{\rm{HT}}$. In Section 3, we will prove the $\Pi^3_0$-unsolvability for $\text{\rm{HT}}$ with weak apartness (defined in Section 2). In section 4, we will prove the $\Delta_3^0$-unsolvability for $\text{\rm{HT}}^{\leq 3}$ and $\Pi_3^0$-unsolvability for $\text{\rm{HT}}$ without weak apartness.

	\section{Proof of $\Delta_3^0$ unsolvability}
	In this section, we will construct a coloring function $c$ in two colors which is not a solution to any infinite $\Delta_3^0$ set. First, we list some notations. Some of them are from Blass, Hirst, and Simpson's paper\cite{bhs} and Hindman's paper\cite{HINDMAN19741}.

\begin{definition}
	
	1. Define two functions $\mu(x),\lambda(x)$ on $\mathbb{Z^+}$ as follows:
	
	For $x=\sum_{i=0}^{k}2^{n_i}\in\mathbb{Z^+}$ where $n_0<\cdots <n_k$, define $\mu(x)=n_k,\lambda(x)=n_0$.

	2. For two positive integers $x<y$, we write that $x\ll y$ or $y\gg x$ if $\mu(x)<\lambda(y)$.  
	A set $A$ has apartness (or is with apartness) if for any $x<y$ in $A$, $x\ll y$.

	3. For each natural number $n$, define the set $B^n=\{x\in \mathbb{Z^+}:\mu(x)=n\}$ and $B^{\leq n}=\{x\in \mathbb{Z^+}: \mu(x)\leq n\}$.
	
	4. A set $A$ has weak apartness if for any natural number $m$, $B^m\cap A$ contains at most one element and for any natural number $l$, there are at most two numbers $x$ such that $\lambda(x)=l$.
	
	Hence if a set $A$ has weak apartness, then for any natural number $L$, there exists $M$ such that for any $x>M$ in $A$, $\lambda(x)>L$. 
\end{definition}

	For convenience, for a set $A$, we use $A(x)$ to denote its characteristic function. We say that a coloring $c$ or a construction of coloring kills a set $A$, or a set $A$ is killed, if $FS(A)$ is or will not be monochromatic (under $c$). And in this paper, any quantifier without specified range is over the set of natural numbers $\mathbb{N}$, and the same for intervals. For example, $\forall x \phi$ means that for any $x\in\mathbb{N}$, $\phi$ holds and $n\in(a,b),n\in(a,b],n\in[a,b)$ means that $n$ is an integer between $a$ and $b$, not a real number. And for all limit notations such as $\lim_x f(x)$ or $\lim\limits_{x}f(x)$, it means: $\lim\limits_{x\rightarrow\infty}f(x)$. A number $x\gg Y$ for a set $T$ means for any $y\in Y$, $x\gg y$.

	\subsection*{Strategy}

A naive strategy is to set infinitely many requests $\langle x,y\rangle\in\mathbb{N}^2$ and each such request means that let $x$ be the Red and $y$ be Green. But this strategy does not succeed. Instead, we try another kind of request. A request $\langle w,x\rangle$ means that let $w+x$ and $w$ be different colors. We must limit the number of requests we set. For example, if we set requests $\langle 4,2\rangle,\langle 4,3\rangle$, and $\langle 6,1\rangle$, then because $4+2=6,4+3=7,6+1=7$, all three numbers $4,6,7$ need to have different colors, which contradicts the condition that $c$ has only two colors. 

\begin{example}
		Consider a set of requests $\{\langle w,2^n\rangle:\lambda(w)>n\}$. We want to find a coloring function $c:\mathbb{N}\rightarrow \{0,1\}$ such that for any two numbers $w$,$n$ such that $\lambda(w)>n$ (i.e., $2^{n+1}|w$), $w$ and $w+2^n$ have different colors.  
	
	We can find a simple solution: let $c(w)$ be the number of digits $1$ mod 2 in the binary representation of $w$ : for example, $c(2)=c(10_2)=1$, $c(3)=c(11_2)=0$, $c(4)=c(100_2)=1$, $c(5)=c(101_2)=0$ and so on.
\end{example}

We will construct a recursive function $R(n,w)$ from $\mathbb{N}^2$ to $\mathbb{N}$ and each $\langle w, R(n,w)\rangle$ will be a request, such that:

C1. There exists a recursive coloring function $c:\mathbb{N}\rightarrow \{0,1\}$ such that for any $w$ and $n<\lambda(w)$,  $c(w)\neq c(R(n,w)+w)$.

C2. For every infinite $\Delta_3^0$ set $A$ with weak apartness, there exists $w\in FS(A)$ and $n<\lambda(w)$ such that $R(n,w)$ is in $A$.

The following lemma provides a sufficient condition for $R(n,w)$ to satisfy C1.

\begin{lemma}\label{lem:main0}
		Let $R$ be a function from a recursive subset of $\mathbb{N}^2$ to $\mathbb{N}$ such that $R(n,w)\in B^n$. There exists a coloring function $c:\mathbb{N}\rightarrow \{0,1\}$ uniformly recursive in $R$ such that for any $w$ and $n<\lambda(w)$, if $R(n,w)$ is defined, then
		\begin{equation*}
			c(w)\neq c(R(n,w)+w).
		\end{equation*} 
\end{lemma}

\begin{proof}
	Because the domain of $R(n,w)$ is recursive, we extend $R$ to whole $\mathbb{N}^2$.
	
	We construct the coloring function $c$ stage by stage. At stage $s$, we color numbers in $B^s$. For each $s$, we will create a graph with vertex set $B^s$ and edges depending on the function $R$ such that this graph is a tree. So for any $w_1,w_2\in B^s$, there exists a unique path $P$ from $w_1$ to $w_2$.

	Fix a stage $s$. For each $w\in B^s$ and $n< \lambda(w)$, We add one edge between $w+R(n,w)$ and $w$ and call it $e_{w,w+R(n,w)}$. This graph does not contain any multiple edges or self-loops. We show that this graph is a tree by induction. For every $w\in B^s$ and $n<\lambda(w)$, define the vertex set $T_{w,n}=\{w+x:0\leq x<2^{n+1}\}$ and consider the corresponding induced subgraph. By definition, $T_{w_1,n_1}=T_{w_2,n_2}$ iff $(w_1,n_1)=(w_2,n_2)$.
	
	If we write integers in binary representation, then the vertex set $T_{w,n}$ looks like $$\hspace*{-10ex}T_{w,n}=\{(w_sw_{s-1}\cdots w_{n+1}x_n\cdots x_{0})_2:x_i\in\{0,1\}\},$$	where $w_sw_{s-1}\cdots w_{n+1}$ is the highest $s-n$ many digits of $w$ and it is also the common part of all numbers in the vertex set $T_{w,n}$. All other digits of $w$ are $0$. And by definition, for $n\geq 1$,\begin{align*}
		T_{w,n}=&\{w+x:0\leq x<2^{n+1}\}\\
		=&\{w+x:0\leq x<2^{n}\}\cup \{w+x:2^n\leq x<2^{n+1}\}\\
		=&\{w+x:0\leq x<2^{n}\}\cup \{w+2^n+x:0\leq x<2^{n}\}\\
		=&T_{w,n-1}\cup T_{w+2^n,n-1}.
	\end{align*}
	The two sets $T_{w,n-1}$ and $T_{w+2^n,n-1}$ look like\begin{align*}
		T_{w,n-1}&=\{(w_sw_{s-1}\cdots w_{n+1} 0 x_{n-1}\cdots x_{0})_2:x_i\in\{0,1\}\}\\
		T_{w+2^n,n-1}&=\{(w_sw_{s-1}\cdots w_{n+1} 1 x_{n-1}\cdots x_{0})_2:x_i\in\{0,1\}\},
	\end{align*}Hence the induced subgraph of $T_{w,n}$ is the union of the induced subgraphs of $T_{w,n-1},T_{w+2^n,n-1}$ and the edges between vertices in $T_{w,n-1}$ and vertices in $T_{w+2^n,n-1}$. We claim that there exists exactly one edge between two vertex sets $T_{w,n-1},T_{w+2^n,n-1}$.
	
	Assume there exists an edge $e'$ from $w'\in T_{w,n-1}$ to $w'+R(n',w')\in T_{w+2^n,n-1}$. Because by definition $\mu(R(n',w'))=n'<\lambda(w')$, if $n'<n$ then $w'+R(n',w')$ is still in $T_{w,n-1}$ and if $n'>n$, then $w'+R(n',w')$ is greater than all numbers in $T_{w+2^n,n-1}$. Hence $n'=n$. Again, because $n=n'<\lambda(w')$, $w'=w$ otherwise $w'$ is not in $T_{w,n-1}$. Hence such edge is unique and the claim is true.

%
%
%
	
	
	Now we show that all induced graphs of $T_{n,w}$ are trees by induction on $n$. For $n=0$ and every $w$ that $\lambda(w)>0$, $T_{w,0}$ has a single vertex. For $n>0$, by induction hypothesis, for any $w$ that $\lambda(w)>n$, $T_{w,n}$ is a tree because it is the union of two trees (induced graphs of $T_{w,n-1}$ and $T_{w+2^n,n-1}$) plus one edge between two trees. 
	
	Since $B^s=T_{2^s,s-1}$, the graph of $B^s$ is a tree.

	Definition of the coloring: Let $c(2^s)=0$. For another vertex $w$, because the graph is a tree, find the unique path $P$ from $2^{s}$ to $w$ and let $c(w)$ be the length of $P$ modulo 2. This definition is uniformly recursive in $R$.

	Verification: Let $P_w$ be a path from $2^s$ to $w$ and $x=R(n,w)$. If $w+x$ is a vertex in $P_w$, then it must be the vertex in $P_w$ exactly before $w$: otherwise there will be two paths from $w$ to $w+x$. So $c(w+x)\equiv|P_w|-1\!\!\mod2 \neq c(w)$.   If $w+x$ is not in the path, then the path $P_w$ plus $e_{w,w+x}$ is a path from $2^s$ to $w+x$, and $c(w+x)\equiv|P_w|+1\!\!\mod2 \neq c(w)$. In both cases, $c(w)\neq c(w+x)$, and this coloring satisfy the requirement of the lemma.
\end{proof}

	Once we have Lemma \ref{lem:main0}, our task becomes to construct a function $R$ satisfying the second condition C2. The construction will use approximations of $\Delta_3^0$ sets. By recursion theory \cite{Soare1999}, a set $A$ is $\Delta_3^0$ iff $A\leq_T \emptyset''$ and $A\leq_T B'$ iff there exists a function $\tilde{A}(x,k)\leq_T B$ such that $\lim_k \tilde{A}(x,k)=A$. More precisely, we will use the following lemma.

\begin{lemma}\label{prop:limit1}
	There exists a recursive list of total recursive functions $\{A_i(x,k,s)\}_{i\in\mathbb{N}}$ valued in $\{0,1\}$ such that for any $\Delta_3^0$ set $A$ and its characteristic function $A(x)$, there exists a function $A_i(x,k,s)$ in the list such that for any $x$, $$\lim_{y}\lim_s A_i(x,k,s)=A(x).$$
\end{lemma}

\begin{proof}
	 If a set $A$ is $\Delta_3^0$,  then $A\leq_T \emptyset''$ and hence there exists a function $\tilde{A}(x,k)\leq_T \emptyset'$ such that $\lim_k \tilde{A}(x,k)=A(x)$. Let $\tilde{A}(x,k)=\Phi^{\emptyset'}_i(x,k)$, where $\Phi_i$ is the $i$ th recursive function with oracle. Consider $\Phi_i^{\emptyset'[s]}(x,k)[s]$ and we will show that for every $x$,  $\lim_{k}\lim_s \Phi_i^{\emptyset'[s]}(x,k)[s]$ exists and equals $A(x)$ .

	By definition, there exists $k_0$ such that for any $k>k_0$, $\tilde{A}(x,k)\downarrow=A(x)$. Now for a number $k>k_0$, we check if $\lim_s\Phi_i^{\emptyset'[s]}(x,k)[s]=\tilde{A}(x,k)$. Because $\tilde{A}(x,k)=\Phi_i^{\emptyset'}(x,k)$ is defined, there exists $s$ such that $\Phi^{\emptyset'[s]}(x,k)[s]\downarrow =\tilde{A}(x,k)$ and the usage of $\emptyset'$ in the computation will never be changed since stage $s$. Now we could see for any $s'>s$, this computation still holds and $\Phi^{\emptyset'[s']}(x,k)[s']=\tilde{A}(x,k)$. So $\lim_s\Phi^{\emptyset'[s]}(x,k)[s]=\tilde{A}(x,k)$. Hence $\lim_{k}\lim_s \Phi_i^{\emptyset'[s]}(x,k)[s]=A(x)$ .

	So we define the list as follows. For each natural number $j$,  define $A_j(x,k,s)=\Phi_j^{\emptyset'[s]}(x,k)[s]$ if it is $0$ or $1$, and $A_j(x,k,s)=0$ otherwise. This is a total recursive function and this list is also recursive.
\end{proof}

	\subsection*{Strategy of killing one set} 

Assuming $A$ is an infinite $\Delta_3^0$ set with weak apartness, we show how to kill $A$ by constructing the function $R$. If for this $R$, condition C2 holds, then, by Lemma \ref{lem:main0}, we can construct a coloring function $c$ that kills the set $A$.

Our strategy is to guess which number is in $A\cap B^n$ (by weak apartness, $A\cap B^n$ contains at most one element) and let $R(n,w)$ be this element. We can use its approximation $A(x,k,s)$ as a guess function.

  We use the guess function to construct the function $R(n,w)$. For numbers $n,w$, compute $A(x,\lambda(w),\mu(w))$ for all $x\in B^n$, and check if there exists any $x$ in $B^n$ such that $$A(x,\lambda(w),\mu(w))=1,$$ which suggests that $x\in B^n\cap A$; if such $x$ exists, then we guess $B^n\cap A\neq\emptyset$ and let $R(n,w)$ be the minimal one among all such $x$; if such $x$ does not exist, let $R(n,w)=2^{n+1}-1$, i.e.: $$R(n,w)=\min(\{x\in B^n:A(x,\lambda(w),\mu(w))=1\}\cup\{2^{n+1}-1\}).$$

	We show that this construction can kill set $A$. Because $A$ is infinite, there exists $n$ such that $B^n\cap A$ is nonempty. Let $x$ be the unique element in $B^n\cap A$. By Lemma \ref{prop:limit1}, for any number $y$ in $B^n$, $\lim_k\lim_s A(y,k,s)=A(y)$. Because $B^n$ is finite, 
	$\lim_k\lim_s (A(\cdot,k,s)\uparrow B^n)=A(\cdot)\uparrow B^n$. Therefore, for sufficiently large $k,s$, for any $y\in B^n\setminus\{x\}$, $A(y,k,s)=0$ and $A(x,k,s)=1$.

	Now, look for a number $K$ such that for any $k>K$, $\lim_s (A(\cdot,k,s)\uparrow B^n)=A(\cdot)\uparrow B^n$ and choose a number $w_1\in A$ such that $\lambda(w_1)>K$ and $n$; let $k=\lambda(w_1)$. Next, for the fixed number $k$, look for a number $s_k$ such that for any $s>s_k$, $A(\cdot,k,s)\uparrow B^n=A(\cdot)\uparrow B^n$ and choose another number $w_2\gg w_1$ from $A$ such that $\mu(w_2)>s_k$. 

Let $w=w_1+w_2$, then $\lambda(w)=\lambda(w_1)=k>K$ and $\mu(w)=\mu(w_2)>s_k$. Combining all discussions above, $A(\cdot,\lambda(w),\mu(w))\uparrow B^n=A(\cdot)\uparrow B^n$. This means $A(x,\lambda(w),\mu(w))=1$ while for any $y\in B^n\setminus\{x\}$, $A(y,\lambda(w),\mu(w))=0$. By the definition of function $R(n,w)$, $R(n,w_1+w_2)=x\in A$. All three numbers $x\ll w_1\ll w_2$ are in $A$.	By Lemma \ref{lem:main0}, there exists a recursive function $c$ such that $c(w_1+w_2)\neq c(x+w_1+w_2)$ where both $w_1+w_2$ and $x+w_1+w_2$ are in $FS(A)$.

	\subsection*{Kill infinitely many sets}

The strategy of killing infinitely many sets is based on the strategy of killing a single set.   Now, condition C2, as defined earlier, becomes: 	for any $\Delta_3^0$ infinite set $A_i$ with weak apartness, there exist numbers $w\in FS(A_i)$ and $n<\lambda(w)$ such that $R(n,w)\in A_i$. We will continue to use the construction of previous section with one more step: first of all, we need to choose a set $A_i$ among all $\Delta_3^0$ sets by some priority arguments, and then apply that construction. To do this, we first guess which $A_i$ have nonempty intersection with $B^n$, and then choose one such set among them. This argument needs the second guessing function guessing that: for a number $n$ and a set $A_i$, whether $A_i\cap B^n$ is empty or not. For a set $A_i$, a number $w$, and a number $n<\lambda(w)$, we check if there exists $x\in B^n$ such that $A_i(x,\lambda(w),\mu(w))=1$ and use a function $a_i(n,k,s)$ to denote this check: 

\begin{align*}
	a_i(n,k,s)=\max\{A_i(x,k,s):x\in B^n\}=\left\{
	\begin{array}{l}
		1, \text{ if }\exists x\in B^n\ A_i(x,k,s)=1 \\
		0, \text{ otherwise} \\
	\end{array}
	\right..
\end{align*}

	Because $B^n=\{x:\mu(x)=n\}$ is finite, this function is also recursive and total. If $a_i(n,\lambda(w),\mu(w))=1$, which means that such $x\in B^n$ that $A_i(x,\lambda(w),\mu(w))=1$ exists, then we guess $A_i\cap B^n$ is nonempty; otherwise, we guess $A_i\cap B^n$ is empty.

We cannot recursively list all infinite $\Delta_3^0$ sets, but by Proposition \ref{prop:limit1}, we have a recursive list of some total recursive functions $A_i(x,k,s)$ such that every $\Delta_3^0$ set is represented in the list. This list contains approximations of all infinite $\Delta_3^0$ sets. However, some functions in the list may not correspond to an infinite $\Delta_3^0$ set, a set with weak apartness or even a set. 

So we need to use some restrictions similar to Blass, Hirst and Simpson's paper. During the construction, for any function $A_i(x,k,s)$ and any pair $k,s$, we will choose at most $2^i$ many numbers $n\in(i,s)$ and call them the candidates of $A_i$ for $k,s$. For any $k,s$, any set $A_i$ has at most $2^i$ many candidates. When defining $R(n,w)$, only if $n$ is a candidate of $A_i$ for $k=\lambda(w),s=\mu(w)$, $A_i$ may be chosen to apply the construction of killing one set.

\subsection*{Construction} We use the recursive list $\{A_i(x,k,s)\}_{i\in\mathbb{N}}$ in Proposition \ref{prop:limit1}. First, for each $A_i(x,k,s)$ in the list, define  $a_i(n,k,s)=\max\{A_i(x,k,s):x\in B^n\}$ as mentioned before.

 Next, for each $k,s$ define a candidate set $C_i(k,s)$ to be the set of first $2^i$ numbers $n\in(i,s)$ such that $a_i(n,k,s)=1$. If there are not $2^i$ many such number $n$, then simply let $C_i(k,s)$ be all such $n$. These sets $C_i(k,s)$ are uniformly recursive in $i,k,s$.

Third, for each $n,w$, compute the set $\{i< n: n\in C_i(\lambda(w),\mu(w))\}$, the index set of sets $A_i$ that $C_i(\lambda(w),\mu(w))$ contains $n$, which means that we guess $A_i\cap B^n\neq \emptyset$.

Finally, let $j=\min (\{i< n: n\in C_i(\lambda(w),\mu(w))\}\cup\{n\}))$ and $$R(n,w)=\min(\{x\in B^n:A_j(x,\lambda(w),\mu(w))=1\}\cup\{2^{n+1}-1\})$$ as killing one set in previous subsection. This is the construction of the function $R(n,w)$. $R(n,w)$ is a recursive function.

By Lemma \ref{lem:main0}, we can construct a coloring function $c$ recursively in $R(n,w)$ and that coloring function $c$ satisfies the condition C1: for any $w$ and $n<\lambda(w)$, $c(w)\neq c(w+R(n,w))$. The only remaining task is verifying condition C2 for $R(n,w)$ and $c$.

	\subsection*{Verification}

In this subsection, we show that the above construction satisfies condition C2. By Proposition \ref{prop:limit1}, every $\Delta_3^0$ set is contained in the list, we only need to consider sets in the list. We use $A_i$ to denote the set represented by $A_i(x,k,s)$ if that set is defined.  First, we have the following lemma.

\begin{lemma}\label{lem:oldlem2}
	Let $C_i(k,s)$ be the function in our construction. For any infinite $\Delta_3^0$  set $A_i$ with weak apartness, there exists a finite set $C_i$ of size $2^i$ such that $\lim_k\lim_sC_i(k,s)\downarrow=C_i$.
\end{lemma}

\begin{proof}
	Because $A_i$ is infinite and has weak apartness, there are infinitely many $n>i$ such that $A_i\cap B^n$ is nonempty. Let $N$ be the $2^i$ th number $n$ among them.

	By definition, for any $x$ in $B^{\leq N}$, $\lim_{k}\lim_s A_i(x,k,s)=A_i(x)$. Because $B^{\leq N}$ is finite, by properties of limits, we also have: $$\lim_{k}\lim_s (A_i(\cdot,k,s)\uparrow B^{\leq N})=A_i(\cdot)\uparrow B^{\leq N}$$

	Hence for any sufficiently large $k$, there exists $s_k$ such that for any $s>s_k$ and every $n\leq N$, $a_i(n,k,s)=\max\{A_i(x,k,s):x\in B^n\}=\max\{A(x):x\in B^n \}$. This means, for $n\leq N$, $a_i(n,k,s)=1$ iff $B^n\cap A_i$ is nonempty . Because $N$ is the $2^i$th number $n>i$ such that $B^n\cap A_i$ is nonempty, there exists $2^i$ numbers $n\leq N$ greater than $i$ such that $B^n\cap A_i$ is nonempty and hence $a_i(n,k,s)=1$. By definition, $C_i(k,s)$ is the set of first $2^i$ elements $n\in(i,s)$ such that $a_i(n,k,s)=1$; hence for those sufficiently large $k,s$, $C_i(k,s)$ is equal to the set of first $2^i$ number $n>i$ that $B^n\cap A_i\neq\emptyset$.

	So we can conclude that, for any sufficiently large $k$, there exists $s_k$ such that for any $s>s_k$, $C_i(k,s)$ exactly contains first $2^i$ number $n>i$ such that $B^n\cap A_i\neq\emptyset$.
	This is just $\lim_{k}\lim_s C_i(k,s)\downarrow=C_i$ for a fixed set $C_i$.
	\end{proof}

	Next, we have the following lemma:

\begin{lemma}\label{lem:c1s2:final2}
	Let $R(n,k)$ be the function in our construction. For any $\Delta_3^0$ infinite set $A_i$ with weak apartness, there exists a finite set $X_i\subset A_i$ such that: \begin{align*}
		\exists K\forall k>K\exists s_k\forall s>s_k\forall w\Big(\lambda(w)=k\wedge\mu(w)=s\longrightarrow
		\exists x\in X_i (R(\mu(x),w)=x)\Big).
	\end{align*}
	
\end{lemma}

\begin{proof}
	Consider limits $\lim_k\lim_sC_i(k,s)\downarrow=C_i$ and $\lim_{k}\lim_s A_i(\cdot,k,s)\uparrow B^{\leq N}=A_i(\cdot)\uparrow B^{\leq N}$ in the proof of Lemma \ref{lem:oldlem2} and let $X_i=\cup_{n\in C_i}B^n\cap A_i$. Using these two limits, there exists a number $K$ such that for any $k>K$, there exists a number $s_k$ such that for any $s>s_k$, $C_i(k,s)\downarrow=C_i$ and $A_i(\cdot,k,s)\uparrow B^{\leq N}=A_i(\cdot)\uparrow B^{\leq N}$. We show that if these two equations hold for numbers $k,s$, then for any $w$ such that $k=\lambda(w),s=\mu(w)$, there exists $x\in X_i$ such that $R(\mu(x),w)=x$ and hence this lemma is true.
	
	We claim that for these $k,s$, for at least one $n\in C_i(k,s)=C_i$,$$i=\min(\{j< n: n\in C_j(k,s)\}\cup\{n\})).$$ 
	
	By definition, for any $j<i$, $|C_j(k,s)|\leq 2^j$. So $|\cup_{j<i}C_j(k,s)|<2^i$.  If the claim is not true, then for every $n\in C_i(k,s)$, there exists $j<i$ that $n\in C_j(k,s)$. This means that $\cup_{j<i}C_j(k,s)\supseteq C_i(k,s)$, which is impossible because $|\cup_{j<i}C_j(k,s)|<2^i=|C_i(k,s)|$. So the claim is true.

	Let $n\in C_i$ be a number such that $i=\min(\{j<n: n\in C_j(k,s)\}\cup\{n\})).$ For any $w$ such that $\lambda(w)=k\wedge\mu(w)=s$, by definition, $$R(n,w)=\min(\{x\in B^n:A_i(x,\lambda(w),\mu(w))=1\}\cup\{2^{n+1}-1\}).$$ Because for $k,s$, $A_i(x,k,s)=A_i(x)$ for all $x\in B^n$, $\min\{x\in B^n: A_i(x,k,s)=1\}$ is defined and is just the unique element in $B^n\cap A_i$. Hence $R(n,w)$ is the unique element in $B^n\cap A_i$ and is in $X_i$.
	
	So for these $k,s$ that $C_i(k,s)\downarrow=C_i$ and $A_i(\cdot,k,s)\uparrow B^{\leq N}=A_i(\cdot)\uparrow B^{\leq N}$, for any $w$ such that $\lambda(w)=k,\mu(w)=s$, there exists $x\in X_i$ such that $R(\mu(x),w)=x$.  This finishes the proof.
	\end{proof}

	Now we can get the following result.

\begin{theorem}\label{prop:result0}
	There exists a recursive coloring function $c$ such that for any infinite $\Delta_3^0$ set $A_i$ with weak apartness, there exist $x\ll w_1\ll w_2$ in $A_i$ such that $c(w_1+w_2)\neq c(x+w_1+w_2)$. Hence $FS^{\leq 3}(A_i)$ is not monochromatic under $c$.
\end{theorem}

\begin{proof}
	Choose the coloring function $c$ in our construction. For an infinite $\Delta_3^0$ set $A_i$ with weak apartness, we apply Lemma \ref{lem:c1s2:final2}  for $A_i$ and get $X_i\subseteq A_i$ , and $K$.
	
	Choose a sufficiently large number $w_1\gg X_i$ in $A_i$ such that $\lambda(w_1)>K$; let $k=\lambda(w_1)$. For this fixed $k$, find $s_k$ by Lemma \ref{lem:c1s2:final2} and choose a sufficiently large $w_2\gg w_1$ in $A_i$ such that $\mu(w_2)>s_k$; let $s=\mu(w_2)$. Because $w_2\gg w_1$, $\lambda(w_1+w_2)=\lambda(w_1)=k,\mu(w_1+w_2)=\mu(w_2)=s$. Hence by Lemma \ref{lem:c1s2:final2}, there exists a number $x\in X_i$ such that $R(\mu(x),w_1+w_2)=x$. By the construction and Lemma \ref{lem:main0}, $c(w_1+w_2)\neq c(x+w_1+w_2)$.
\end{proof}

	Theorem \ref{prop:result0} shows that there exists a recursive coloring function $c$ such that for any infinite $\Delta_3^0$ set $A$ with weak apartness, $FS^{\leq 3}(A)$ is not monochromatic. This theorem is close to our aim. We can use the following fact proved by Hindman:

\begin{lemma}[\cite{HINDMAN19741}]
	For any infinite set $A$, there exists an infinite set $B\leq_T A$ such that $B$ has apartness and $FS(B)\subseteq FS(A)$. 
\end{lemma}

So for the same coloring function $c$, if it has any infinite $\Delta_3^0$ set $A$ as a solution, then it can compute an infinite solution $B$  with apartness. Since $A$ is $\Delta_3^0$ and $B\leq_T A$, $B$ is also $\Delta_3^0$. However, Theorem \ref{prop:result0} implies that this is impossible. Now we have proved our aim of this section as follows:

	\begin{theorem}
	There exists a recursive coloring function $c$ without any $\Delta_3^0$ solution for Hindman's theorem. 
\end{theorem}

\section{Proof of $\Pi_3^0$ unsolvability with weak apartness}

\subsection*{General Preparation}To prove $\Pi_3$ case, we will use more functions similar to $R(n,w)$, denoted as $R,Q$. And we need a lemma generalizing Lemma \ref{lem:main0}.

\begin{lemma}\label{lem_main}
	Let $R(n,w)$ be a function with recursive domain such that $R(n,w)\in B^n$, and  $r>1$. There exists a coloring function $c:\mathbb{N}\rightarrow r$ uniformly  recursive in $R$ such that for any $w$ and any $n<\lambda(w)$, if $R(n,w)$ is defined, then:
	\begin{align*}
		\ c(w+R(n,w))\equiv c(w)+1 \mod r.
	\end{align*}		 
\end{lemma}

\begin{proof}
	This proof is very similar to the proof of Lemma \ref{lem:main0}. First, we extend the domain of $R$ to whole $\mathbb{N}^2$.
	
	For each number $s$, we create a graph with vertex set $B^s$ as in Lemma \ref{lem:main0}. As shown in Lemma \ref{lem:main0}, this graph is a tree, meaning that there exists a unique path from one vertex to another. 
	
	Definition of the coloring: Let $c(2^s)=0$. For another vertex $w$, because the graph is a tree, find the unique path $P$ from $2^{s}$ to $w$.  An edge $e_{w',w'+R(n',w')}$ in the path is considered a positive edge if the direction from $w'$ to $w'+R(n',w')$ is the same as the direction from $2^{s}$ to $w$; otherwise it is considered a negative edge. Then we count the number of positive edges, $a$ and the numbers of negative edges, $b$ and let the color of $P$ be $a-b\mod r$. This process is uniformly recursive in $R$.
	
	Verification: similar to Lemma \ref{lem:main0}, let $P_w$ be the unique path from $2^s$ to $w$, and let $x=R(n,w)$. If $w+x$ is a vertex in $P_w$, then it must be the vertex in $P_w$ exactly before $w$; and because the direction of edge $e_{w,w+x}$ is from $w$ to $w+x$, $P_w$ is $P_{w+x}$ adding a negative direction edge.  So $c(w)\equiv c(x+w)-1\mod r$. If $w+x$ is not in the path, then the path $P_w$ adding $e_{w,w+x}$ is a path from $2^s$ to $w+x$; because the direction of edge $e_{w,w+x}$ is from $w$ to $w+x$, $P_{w+x}$ is $P_{w}$ adding a positive direction edge.  So $c(w)+1\equiv c(x+w)\mod r$.
\end{proof}

We will use two special forms of this lemma. The first special form is when $r=2$, it returns to Lemma \ref{lem:main0}. The second special form is when the function $R(n,w)$ only depends on $\lambda(w)$, $\mu(w)$ in addition to $n$: for example, the function in the last section, $$R(n,w)=\min(\{x\in B^n:A_j(x,\lambda(w),\mu(w))=1\}\cup\{2^{n+1}-1\}).$$ We can see, for $w_1\neq w_2$, if $\lambda(w_1)=\lambda(w_2)$ and $\mu(w_1)=\mu(w_2)$, then $R(n,w_1)=R(n,w_2)$. In this case, we write this kind of functions in a form of three variables like $R(n,k,s),Q(n,k,s)$.

\begin{corollary}\label{col:pre1}
	Let $Q$ be a function from a recursive subset of $\mathbb{N}^3$ to $\mathbb{N}$ such that  $Q(n,k,s)\in B^n$ and $r>1$, there exists  a coloring function $c:\mathbb{N}\rightarrow r$ uniformly recursive in $Q$ such that for any $w$ and $n<\lambda(w)$, if $Q(n,\lambda(w),\mu(w))$ is defined, then: $$\ c(w+Q(n,\lambda(w),\mu(w)))\equiv c(w)+1 \mod r.$$
\end{corollary}

 Similar to last section, we need to choose a list approximation of $\Pi_3^0$ sets.
  By definition, every $\Pi_3^0$ set $A$ can be represented as $x\in A\leftrightarrow\forall y \Phi(x,y)$ by a $\Sigma_2^0$ formula $\Phi(x,y)$. By recursion theory \cite{Soare1999}, there exists a total recursive function $g(i,x,y)$ such that if  $\Phi_i(x,y)$ is the $i$th $\Sigma_2^0$ formula, then $\Phi_i(x,y)$ holds iff $W_{g(i,x,y)}$ is finite, which is also equivalent to that $\lim_{s} |W_{g(i,x,y),s}|$ is finite.

So we consider a function $$F(i,x,y,s)=\max\limits_{y'\leq y} |W_{g(i,x,y'),s}|,$$ the max value of sizes of $W_{g(i,x,y')}$ at stage $s$ for all $y'\leq y$. If $x\in A$, then because for any $y$, $W_{g(i,x,y)}$ is finite, $F(i,x,y,s)\leq \max_{y'\leq y} |W_{g(i,x,y'),s}|$ and hence $\lim_s F(i,x,y,s)$ exists and is finite; if $x\notin A$, then there exists $y_0$ such that $|W_{g(i,x,y_0)}|=\infty$ and $\lim_s |W_{g(i,x,y_0),s}|=\infty$, and so for any $y\geq y_0$, $\lim_s F(i,x,y,s)=\infty$. We conclude this by the following lemma:

\begin{lemma}\label{lem:list}
	There exists a total recursive function $F(i,x,y,s)$ such that:
	
	1.The function $F(i,x,y,s)$ is non-decreasing with respect to $y$ and $s$ separately. Therefore for any $i,x$ and $y$, the limit $\lim_s F(i,x,y,s)$ exists  and this limit is non-decreasing with respect to $y$.

	2. For each $i$, let $A_i=\{x:\forall y\ \lim_s F(i,x,y,s)\text{ is finite}\}$. 	Then $A_i$ is a $\Pi_3^0$ set.

3. For any $\Pi_3^0$ set $A$, there exists $i$ such that $A=A_i$ where $A_i$ is defined in 2.



%
%
%

\end{lemma}

Therefore, the function $F(i,x,y,s)$ defines a list $\langle A_i\rangle_{i\in\mathbb{N}}$ of $\Pi^0_3$ sets that includes all $\Pi_3^0$ sets.

\subsection*{Strategy}

As in the case of  $\Delta_3^0$, to define the recursive coloring function $c$, we will construct a recursive function $R(n,w)$ and apply Lemma \ref{lem_main}. Similar to $\Delta_3^0$ case, the function $R(n,w)$ must satisfy the condition that, for any infinite $\Pi_3^0$ set $A_i$ with weak apartness, there exists $w\in FS(A_i)$ and $n\ll w$ such that $R(n,w)\in A_i$. Constructing such a function $R(n,w)$ is the key part of the strategy.

During the construction in the $\Delta_3^0$ case, for each $n$, we guess and choose a correct set $A_i$ by a priority argument such that $A_i\cap B^n\neq\emptyset$. We then try to construct a function $R(n,w)$ such that for sufficiently many $w$, $R(n,w)$ is an element in $A_i$. 


In $\Pi_3^0$ case, we will use new strategy. During the construction, for a fixed $n$, we guess and choose a specific set $A_i$ by a priority argument such that  $A_i\cap B^n\neq\emptyset$, and aim to make it so that the value of $R(n,w)$ can be any number in $B^n$ when $w$ goes through the set $FS(A_i)$. If so, there exists a number $w\in FS(A_i)$ such that $R(n,w)\in A_i$.

In other words, for a fixed $n$, we can regard $R(n,\cdot)$ as a coloring function from $w\in\mathbb{N}$ to one of $|B^n|$ many colors, a number $R(n,w)\in B^n$. Then our aim is to let $FS(A_i)$ includes all colors under the coloring function $R(n,\cdot)$. We can use Corollary \ref{col:pre1} to do this.

For each fixed $n$, to construct such a function $R(n,\cdot)$, we construct a recursive function $Q_n(y,k,s):\mathbb{N}^3\rightarrow \mathbb{N}$ such that $Q_n(y,k,s)\in B^y$ and later we will apply Corollary \ref{col:pre1} for $Q_n(y,k,s)$ and $r=|B^n|$ to get a function $R(n,\cdot)$. The construction of $Q_n(y,k,s)$ depends on the 
choice of the set $A_i$. If we choose a set $A_i$, then our aim is ensure that for most sufficiently large $y,k,s$, $Q_n(y,k,s)$ is in $A_i$.

In this way, we can find enough sufficiently large $w'\in A_i$ such that $w'=Q_n(y',k',s')$ for some $y',k',s'$. And later we are able to find an arbitrarily long sequence of such $w'$: $w_1\ll w_2\ll\dots\ll w_k\ll w_{k+1}\ll\dots\ll w_L$ all in $A_i$ and 
(because of Corollary \ref{col:pre1}) for any $1\leq k<L$\begin{equation*}
	R(n,w_k+w_{k+1}+\cdots+w_L)\equiv R(n,w_{k+1}+\dots+w_L)+1\mod r.
\end{equation*}Therefore, for different $k$, $R(n,w_k+\dots+w_L)$ has different values and if we choose $L>r$, then for any number $x\in B^n$, we can always find $k'$ such that: $R(n,w_{k'}+\dots+w_L)=x$. 


\subsection*{Preparation for construction} We need to guess two things: whether $B^n\cap A_i$ is empty or not and if $B^n\cap A_i$ is nonempty, which element is in $B^n\cap A_i$.

First, we guess the elements in $B^n\cap A_i$ if we assume it is nonempty. If $x\in B^n\cap A_i$, then for any $y$, $\lim_s F(i,x,y,s)$ is finite, and if $x\in B^n$ but $x\notin A_i$, then for almost all $y$, $\lim_s F(i,x,y,s)$ is infinite. So we define the guess function $x(i,n,y,s)$ to be the element $x\in B^n$ with minimal $F(i,x,y,s)$ value: $$x(i,n,y,s)=\min\{x\in B^n:\forall x'\in B^n F(i,x,y,s)\leq F(i,x',y,s)\}.$$

\begin{lemma}[Guessing a single set]\label{lem:single}
	Let $n$ be a natural number and $A_i$ be an infinite $\Pi_3^0$ set with weak apartness in the list such that $A_i\cap B^{n}$ is nonempty, then $\lim_{y}\lim_s x(i,n,y,s)$ exists and equals the unique element in $A_i\cap B^n$.
	
%
%
\end{lemma}

\begin{proof}
	Since $A_i\cap B^n\neq\emptyset$, we choose $x_0\in A_i\cap B^n$. Becasue $A_i$ has weak apartness, $B^n\setminus A_i=B^n\setminus\{x_0\}$. 
	So for any $x\in B^n\setminus\{x_0\}$ 
	there exists $y_x$ that $\lim_s F(i,x,y,s)=\infty$ for any $y>y_x$. Let $y_0$ be the largest one of all such $y_x$.
	
	For any $y>y_0$, there exists $k_y$ that $\lim_s F(i,x_0,y,s)=k_y$. 
	And for each $x\in B^n\setminus\{x_0\}$,  $ \lim_sF(i,x,y,s)>k_y$. So for any $y>y_0$ and any sufficiently large number $s$,  $x(i,n,y,s)=x_0$.
\end{proof}

Next we define a function to guess whether $A_i\cap B^n$ is empty or not.

\begin{lemma}
	For every set $A_i$ in the list, let \begin{equation*}
		B(i,n,y,s)=\min\{F(i,x,y,s):x\in B^n\}.
	\end{equation*} 
	
	Then if $A_i\cap B^n$ is nonempty, then $\forall y \lim_s B(i,n,y,s)<\infty$; and if $A_i\cap B^n$ is empty, then $\exists y_0\forall y>y_0 \lim_s B(i,n,y,s)=\infty$.
\end{lemma} 

\begin{proof}
	By Lemma \ref{lem:list} and properties of finite operations of limits
\end{proof}

 We use the function $B(i,n,y,s)$ to guess whether $A_i\cap B^n$ is empty or not: for a number $k$, if $B(i,n,y,s)<k$, we guess the $A_i\cap B^n$ is nonempty with respect to $k$. This is the second guess function. 
 
 \subsection*{Construction}

We will define a function $I(n,y,k,s)$ which indicates index of the choice $A_i$  mentioned earlier.

\noindent\textbf{At stage $s>0$\ :} For all $n,i,y<s$, compute $B(i,n,y,s)$.

\textbf{At sub-stage $n\in(0,s)$:}

For all $y,k$ that $n<y\leq k\leq s$, do the following steps.

\hspace*{2ex}\textbf{Step 1:} Define $$I(n,y,k,s)=\min(\{i<n: B(i,n,y,s)<k \text{ and }\forall m<n\ I(m,y,k,s)\neq i\}\cup\{\infty\}).$$

\hspace*{2ex}\textbf{Step 2:} Define $Q_n(y,k,s)$: if $I(n,y,k,s)\neq \infty$, let $j=I(n,y,k,s)$, and let $$Q_n(y,k,s)=x(j,y,k,s);$$

\hspace*{2ex}otherwise let $Q_n(y,k,s)=2^{y}$.

\hspace*{2ex}\textbf{Step 3:} If $n+1<s$, go to next substage; otherwise end this stage.

\vspace*{1ex}

 For each $n$, the function $Q_n(y,k,s)$ is recursive with recursive domain. We can apply Corollary \ref{col:pre1} to $Q_n(y,k,s)$ and $r=|B^n|$ to obtain a function $R_0(n,\cdot)$ with range $\{0,\cdots,r-1\}$. Because $R(n,w)$ has a range $B^n$, we define $R(n,w)$ to be the $R_0(n,w)$th number in $B^n$.

 Because the sequence $\langle Q_n(y,k,s)\rangle_{n\in\mathbb{N}}$ is a recursive sequence and Corollary \ref{col:pre1} is uniformly recursive in $Q_n(y,k,s)$, $\langle R(n,\cdot)\rangle_{n\in\mathbb{N}}$ is a recursive sequence and the function $R(n,w)$ is recursive.


\subsection*{Verification, part 1}

In this subsection, we will show that for any infinite $\Pi_3^0$ set $A_i$, there exists $n$ such that $\lim_y\lim_k\lim_s I(n,y,k,s)=i$.

\begin{lemma}\label{lem:lim1}
	
 for any $n$, $\lim_y\lim_k\lim_s I(n,y,k,s)$ is defined.  Moreover, if we let $I(n)=\lim_y\lim_k\lim_s I(n,y,k,s)$, $H(n)=\{i<n:\exists m<n, I(m)=i\}$ and $Z(n)=\{i<n: \exists x\in B^n, x\in A_i\}$, then for any $n$, $I(n)=\min (Z(n)\setminus H(n)\cup\{\infty\})$. 
\end{lemma}

\begin{proof}
	We prove this lemma by induction on $n$. For a number $n$, assume for any $m<n$, $I(m)$ is defined and $I(m)=\min Z(m)\setminus H(m)\cup\{\infty\}$. 
	
	By definition, for each $y,k,s$, $$I(n,y,k,s)=\min(\{i<n:B(i,n,y,s)<k,\forall m<n,I(m,y,k,s)\neq i\}\cup\{\infty\}).$$ The major part of the right-hand side is the set $$\{i<n:B(i,n,y,s)<k,\forall m<n,I(m,y,k,s)\neq i\},$$ which is the set $\{i<n:B(i,n,y,s)<k\}$ minus $\{i<n:\exists m<n,I(m,y,k,s)=i\}$. We claim that \begin{equation}
		\lim_y\lim_k\lim_s \{i<n:B(i,n,y,s)<k\}=Z(n)
	\end{equation} and \begin{equation}
		\lim_y\lim_k\lim_s \{i<n:\exists  m<n,I(m,y,k,s)=i\}=H(n)
	\end{equation}

	To prove equation (3.1), we begin with the definition of $Z(n)$ and $B(i,n,y,s)$. By properties of finite operations of limits, there exists $y_1$ such that for any $y>y_1$ and any $k$, there exists $s'_{y,k}$ such that for any $s>s'_{k,y}$ and any $i\in n\setminus Z(n)$, $B(i,n,y,s)>k$. Similarly, for any $y$ there exists $k''_y$ such that for any $s$ and any $i\in Z(n)$, $B(i,n,y,s)<k''_y$.  Hence: \begin{align*}
		\lim_y\lim_k\lim_s \{i<n:B(i,n,y,s)<k\}&\subseteq Z(n)\\ \lim_y\lim_k\lim_s \{i<n:B(i,n,y,s)<k\}&\supseteq Z(n).
	\end{align*} Equation (3.1) is true.
	
	We will now prove equation (3.2). By the Induction Hypothesis, for any $m<n$, $I(m)$ is defined. Using properties of finite operation of limits, we can find $y_0$ such that for any $y>y_0$, there exists $k_y$ such that any $k>k_y$, there exists $s_{y,k}$ such that for any $s>s_{y,k}$ and any $m<n$, $I(m,y,k,s)=I(m)$. Thus, we have $$\{i<n:\exists  m<n,I(m,y,k,s)=i\}=\{i<n:\exists  m<n,I(m)=i\}.$$ The right-hand side of this equation is equal to $H(n)$. Therefore, we conclude that  $\lim_y\lim_k\lim_s \{i<n:\exists  m<n,I(m,y,k,s)=i\}=H(n)$.

	Since equations (3.1)(3.2) hold, by properties of finite operations of limit, \begin{align*}
		&I(n)=\lim_y\lim_k\lim_s\ I(n,y,k,s)\\
		=&\lim_y\lim_k\lim_s\ \min(\{i<n:B(i,n,y,s)<k,\forall m<n,I(m,y,k,s)\neq i\}\cup\{\infty\})\\
		=&\lim_y\lim_k\lim_s\ \min(\{i<n:B(i,n,y,s)<k\}\setminus\{i<n:\exists m<n,I(m,y,k,s)=i\}\cup\{\infty\})\\
		=&\min(Z(n)\setminus H(n)\cup\{\infty\}).
	\end{align*}
	This is what we want to show.
\end{proof}

\begin{lemma}\label{lem:lim2}
	For any infinite $\Pi_3^0$ set  $A_i$ in our list, there exists $n$ such that $I(n)=i$.
\end{lemma}

\begin{proof}
We prove this lemma by induction on $i$. For each $i$, if $A_i$ is infinite, there exist infinitely many $n$ such that $A_i\cap B^n\neq \emptyset$. For any $j<i$, if $A_j$ is finite, there exists $n_j$ such that for any $n>n_j$, $A_j\cap B^n=\emptyset$; and by induction hypothesis, for any $j<i$, if $A_j$ is infinite, there exists $n_j$ such that $I(m)=j$. So we choose a number $N$ sufficiently large such that $A_i\cap B^N$ is nonempty and for any $j<i$, either $A_j\cap B^N=\emptyset$ or $\exists m<N, I(m)=j$.

	If there exists $n<N$ such that $I(n)=i$, the we are done. Otherwise by the definition of $H(N)$, we have $i\notin H(N)$; since $B^n\cap A_i\neq \emptyset$, we have $i\in Z(N)$.  Therefore $i\in Z(N)\setminus H(N)$. For any $j<i$, if $A_j\cap B^N=\emptyset$, then $j\notin Z(N)$; if $\exists m<N, I(m)=j$ then $j\in H(N)$. Thus for any $j<i$, we have $j\notin Z(N)\setminus H(N)$. Therefore $i=\min(Z(N)\setminus H(N)\cup\{\infty\})$ and by Lemma \ref{lem:lim2}, $I(N)=i$.
	
Therefore, there exists $n\leq N$ such that $I(n)=i$.
\end{proof}

\subsection*{Verification: part 2}

In this subsection, we will finish the verification.

\begin{lemma}\label{lem:final1}
	Let $A_i$ be a $\Pi_3^0$ set with weak apartness and $n$ be a number such that $I(n)=i$. Then for any $r>1$ and $Y$,  there exists a sequence $x_1\ll x_2\ll\cdots\ll x_r$ all in $A_i$ such that:
	
	(1) $\lambda(x_1)>Y$.
	
	(2) for any $g\in(0,r)$, $\lim_sQ_n(\mu(x_{g}),\lambda(x_{g+1}),s)=x_{g}.$
\end{lemma}

\begin{proof}
	Because $I(n)=i$, $\lim_y\lim_k\lim_s I(n,y,k,s)=i$. We will prove this lemma by induction on $r$.

	Because $A_i$ has weak apartness, by Lemma \ref{lem:single}, there exists $Y_0$ such that for any  $Y_1>Y_0$ such that $B^{Y_1}\cap A_i$ is nonempty,  $B^{Y_1}\cap A_i$ contains exactly one element and therefore $\lim_{k}\lim_s x(i,{Y_1},k,s)\in A_i$. Without loss of generality, we assume that $Y>Y_0$. Because $\lim_k\lim_s I(n,Y,k,s)=i$,  $\lim_{k}\lim_s Q_n(Y,k,s)=\lim_{k}\lim_s x(i,Y,k,s)$ is defined and equal to the unique element $x\in A_i\cap B^Y$. 
	
	For the case $r=2$: in this case, $g$ can only be $1$. We choose $x_1\in A_i$ such that $\lambda(x_1)>Y$. Since $ A_i\cap B^{\mu(x_1)}$ contains $x_2$, it is nonempty, and thus $\lim_k\lim_s Q(\mu(x_1),k,s)=x_1$. Then we  choose $k_0$ such that
	for any $k>k_0$, $\lim_s Q(\mu(x_1),k,s)=x_1$ and choose a number $x_2\gg x_1$ in $A_i$ such that $\lambda(x_2)>k_0$. Then \(\lim_s Q(\mu(x_1),\lambda(x_2),s)=x_1\). It follows that two numbers $x_2\gg x_1$ satisfy our requirements.

	For the case $r+1>2$, by the Induction Hypothesis, for $Y$ and $r$, there exists a sequence $x_1\ll\cdots\ll x_r$ that satisfies the lemma for $Y,r$: $$\forall g\in(0,r), \lim_sQ(\mu(x_{g}),\lambda(x_{g+1}),s)=x_{g},$$ and $\lambda(x_1)>Y$. We need to find $x_{r+1}$ such that $\lim_s Q(\mu(x_r),\lambda(x_{r+1}),s)=x_r$. Since $B^{\mu(x_r)}\cap A_i$ is nonempty, $\lim_{k}\lim_s Q_n(\mu(x_r),k,s)=x_r$; therefore we only need to choose  $x_{r+1}$ such that $\lambda(x_{r+1})$ is sufficiently large. Choose a sufficiently large $x_{r+1}\gg x_{r}$ in $A_i$ such that $\lim_s Q_n(\mu(x_r),\lambda(x_{r+1}),s)=x_r$. Now the sequence $x_1\ll x_2\cdots\ll x_{r+1}$ satisfies the lemma.
\end{proof}

\begin{lemma}\label{lem:final3}
	Let $A_i$ be an infinite $\Pi_3^0$ set with weak apartness, $n$ be a natural number such that $I(n)=i$ and $R(n,w)$ be the function constructed. For any $r>1$, there exist $r$ different numbers $w_j, j<r$ in $FS(A_i)$ such that:
	
	(1) for any $w_j$, $\lambda(w_j)>n$;
	
	(2) for any two different $w_{j_1}$ and $w_{j_2}$, $R(n,w_{j_1})\neq R(n,w_{j_2})$.
\end{lemma}

\begin{proof}
	By Lemma \ref{lem:final1}, there exists a sequence $\langle x_g\rangle_{0<g\leq r}$ such that,  $$\forall g\in(0,r), \lim_sQ(\mu(x_{g}),\lambda(x_{g+1}),s)=x_{g}.$$ 
	
	Choose a sufficiently large $x_{r+1}\in A_i$ such that $x_{r+1}\gg x_r$ and $$\forall g\in(0,r), Q(\mu(x_{g}),\lambda(x_{g+1}),\mu(x_{r+1}))=x_{g}.$$
	
	For $ h\in(0,r]$, let $w_h=\sum_{g=h}^{r+1}x_g$. We show that for $0< h_2<h_1\leq r$, $R(n,w_{h_1})\neq R(n,w_{h_2})$. For each $w_h$, $\mu(w_h)=\mu(x_{r+1})$ and $\lambda(w_h)=\lambda(x_h)$. So by Lemma \ref{lem:final1} and our assumption, $Q_n(\mu(x_{h}),\lambda(w_{h+1}),\mu(w_{h+1}))=x_{h}$. Then by the construction and Corollary \ref{col:pre1}:\begin{align*}
		R_0(n,w_{h+1})+1&\equiv R_0(n,Q_n(\mu(x_{h}),\lambda(w_{h+1}),\mu(w_{h+1}))+w_{h+1})\mod r\\
		&\equiv R_0(n,x_{h}+w_{h+1})\mod r\\
		&\equiv R_0(n,w_{h})\mod r.
	\end{align*} So $R_0(n,w_{h_1})\equiv R_0(n,w_{h_2})-(h_1-h_2) \mod r$ and each $w_h$ has distinct $R(n,w_h)$ value.
\end{proof}

\begin{lemma}\label{lem:final4}
	Let $A_i$ be an infinite set with weak apartness and $R(n,w)$ be our function, there exists $n$, $x\in B^n\cap A_i$ and $w\in FS(A_i)$ that $\lambda(w)>n$, such that: \begin{equation*}
		R(n,w)=x
	\end{equation*}

\end{lemma}

\begin{proof}
	By Lemma \ref{lem:lim2}, there exists $n$ such that $I(n)=i$ and by Lemma \ref{lem:final3}, for $r=|B^n|$, there exist numbers $w_i,i<r$ all in $FS(A_i)$ such that $\lambda(w_i)>n$ and each $w_i$ has a distinct $R(n,w_i)$ value. Because $R(n,w_i)\in B^n$ for any such $w_i$, for any $x\in B^n$, there exists one $w_i$ among them such that $R(n,w_i)=x$.
	
	By Lemma \ref{lem:lim1}, $i=I(n)=\min\{Z(n)\setminus H(n)\cup\{\infty\}\}.$ Because $i$ is a number rather than $\infty$, $i\in Z(n)=\{i<n:A_i\cap B^n\neq \emptyset\}$. Hence $A_i\cap B^n\neq\emptyset$. 
	
	Therefore we can choose $x\in B^n\cap A_i$ and $w_i$ such that $R(n,w_i)=x$ and these $n,w_i,x$ are what we need.
\end{proof}

\begin{theorem}\label{thm:c2s3:main1}
	There exists a recursive coloring function $c:\mathbb{N}\rightarrow 2$ such that for any infinite $\Pi_3^0$ set $A_i$ with weak apartness, there exists $w_1,w_2\in FS(A_i)$ such that $c(w_1)\neq c(w_2)$.
\end{theorem}

\begin{proof}
	Use the function $R(n,w)$ in our construction and apply Corollary \ref{col:pre1} for $R(n,w)$ to get a coloring function $c$. We show that $c$ is what we need.
	
	For any infinite $\Pi_3^0$ set $A_i$ with weak apartness, by Lemma \ref{lem:final4}, there exists numbers $n,x,w$ such that $x\in B^n\cap A_i$, $\lambda(w)>n$, $w\in FS(A_i)$ and $R(n,w)=x$. By the Corollary \ref{col:pre1}, $c(w)\neq c(w+R(n,w))$ if $n<\lambda(w)$. So $c(w)\neq c(w+x)$. Because $\lambda(w)>n=\mu(x)$, both $w$ and $w+x$ are in $FS(A_i)$. Two numbers $w,w+x$ are both in $FS(A_i)$ and $c(w)\neq c(w+x)$. This is what we want to show.
\end{proof}

\section{unsolvabilities Without apartness}

In this section, we will prove $\Delta_3^0$ unsolvability for $\text{\rm{HT}}^{\leq 3}$ without weak apartness and $\Pi_3^0$ unsolvability without weak apartness.

\begin{lemma}\label{lem:remove}
	There exists a recursive coloring function $c_0:\mathbb{N}\rightarrow \{0,1\}^2$ such that for any infinite set $A$, if it does not have weak apartness, then it is not a solution to $c$ for $\text{\rm{HT}}^{\leq 2}$.
\end{lemma}

\begin{proof}
	We define two coloring functions $c_1,c_2:\mathbb{N}\rightarrow \{0,1\}$:\begin{align*}
		c_1(x) \equiv \mu(x) \mod 2;\\
		c_2(x) \equiv \lambda(x) \mod 2.
	\end{align*}	We show that for $c_0=c_1\times c_2$, if an infinite set $A$ has not weak apartness, then it is not monochromatic under $c_0$. If $A$ has not weak apartness, there are two cases.
	
	The first case is that, there are two numbers $x_1,x_2$ such that $\mu(x_1)=\mu(x_2)=m$. This means these two numbers are in the interval $[2^m,2^{m+1})$. Thus $\mu(x_1+x_2)=m+1$, and $c_1(x_1)\neq c_1(x_1+x_2)$. Therefore, $A$ is not monochromatic.

	The second case is that, there are three number $x_1,x_2,x_3$ such that $\lambda(x_1)=\lambda(x_2)=\lambda(x_3)=l$. In this case, all three numbers are divisible by $2^l$ and not by $2^{l+1}$. So for any $X$ among these three numbers, $X\equiv 2^l \text{ or }2^{l+1}+2^l\mod 2^{l+2}$ and there must be two number $X_1,X_2$ among $x_1,x_2,x_3$ such that $X_1\equiv X_2 \mod 2^{l+2}$. Hence $X_1+X_2$ is divisible by $2^{l+1}$ but not $2^{l+2}$ and $\lambda(X_1+X_2)=l+1$. Since $\lambda(x_1)=l$, $c_2(X_1+X_2)\neq c_2(x_1)$ and $A$ is not monochromatic.
	
	In each cases, $A$ is not monochromatic. Therefore, if $A$ does not have weak apartness, then it is not monochromatic under $c_0=c_1\times c_2$.
\end{proof}

Using Lemma \ref{lem:remove}, we can prove our theorems.

\begin{theorem}
	There exists a recursive coloring function $c:\mathbb{N}\rightarrow \{0,1\}^3$ such that any infinite $\Delta_3^0$ set $A$ is not a solution to $c$ for $\text{\rm{HT}}^{\leq 3}$.
\end{theorem}

\begin{proof}
	By Theorem \ref{prop:result0}, there exists a recursive coloring $c_1:\mathbb{N}\rightarrow\{0,1\}$ such that any infinite $\Delta_3^0$ set with weak apartness is not a solution for $\text{\rm{HT}}^{\leq 3}$.
	
	Let $c=c_1\times c_0$ where $c_0$ is the coloring in Lemma \ref{lem:remove} and we claim that this coloring is what we need.
	
	For any infinite $\Delta^0_3$ set $A$, if it has weak apartness, then it is not a solution to $c_1$ for $\text{\rm{HT}}^{\leq 3}$; if it does not have weak apartness, then it is not a solution to $c_0$ for $\text{\rm{HT}}^{\leq 2}$ (hence $\text{\rm{HT}}^{\leq 3}$). In each case, $A$ is not a solution to $c=c_1\times c_0$ for $\text{\rm{HT}}^{\leq 3}$. 
\end{proof}

\begin{theorem}
	There exists a recursive coloring function $c:\mathbb{N}\rightarrow \{0,1\}^3$ such that, any infinite $\Pi_3^0$ set $A$ is not a solution to $c$ for $\text{\rm{HT}}$.
\end{theorem}

\begin{proof}
	By Theorem \ref{thm:c2s3:main1}, there exists a recursive coloring $c_2:\mathbb{N}\rightarrow\{0,1\}$ such that any infinite $\Pi_3^0$ set with weak apartness is not a solution to $c_2$ for $\text{\rm{HT}}$.
	
	Let $c=c_2\times c_0$ where $c_0$ is the coloring in Lemma \ref{lem:remove} and we claim that this coloring is what we want.
	
	For any infinite $\Pi^0_3$ set $A$, if it has weak apartness, then it is not a solution to $c_2$ for $\text{\rm{HT}}$; if it does not have weak apartness, then it is not a solution to $c_0$ for $\text{\rm{HT}}^{\leq 2}$ (hence $\text{\rm{HT}}$). So in either case, $A$ is not a solution to $c=c_2\times c_0$ for $\text{\rm{HT}}$. 
\end{proof}

	\section{Open Questions}

Here we list some open questions related to our topic. The first question is to prove the $\Delta_4^0$ case:

\noindent{\bfseries	Question 1}: Is there any recursive coloring without any $\Delta_4^0$ solution for Hindman's theorem?

Similarly, in our paper we prove the $\Delta_3^0$-unsolvability of not only Hindman's theorem, but $\text{\rm{HT}}^{\leq 3}$. However, we have not been able to prove the same theorem for the $\Pi_3^0$ case so far. Therefore, we pose the following question:

\noindent{\bfseries	Question 2}: Is there any recursive coloring such that any infinite $\Pi_3^0$ sets is not a solution for $\text{\rm{HT}}^{\leq 3}$?

Another similar question is about the number of colors although. In this paper we do not focus on this, it is still worth to consider:

\noindent{\bfseries	Question 3}: Is there any recursive 2-coloring such that any infinite $\Delta_3^0$ sets is not a solution for $\text{\rm{HT}}^{\leq 3}$ and is there any recursive 2-coloring  such that any infinite $\Pi_3^0$ sets is not a solution for Hindman's theorm?

%
%
%

Of course, the aim of this type of questions is to find an example without any arithmetic solution. However, there is currently no mainstream consensus on whether the Hindman's theorem can be proven within $\text{\rm{ACA}}_0$. It is also possible that every coloring has an arithmetic or even $\Delta_4^0$ solution. This belongs to the opposite side of reverse math of Hindman's theorem, the upper bound of a solution.


\bibliographystyle{amsplain}
\bibliography{hindman}

\end{document}